\documentclass[10pt,a4wide]{article}

\usepackage{amsfonts}
\usepackage{latexsym}
\usepackage{amsmath}
\usepackage{amsthm}
\usepackage{amssymb}
\usepackage{oldgerm}
\usepackage[]{graphicx}
\usepackage{hyperref}

\newtheorem{theorem}{Theorem}[section]
\newtheorem{corollary}[theorem]{Corollary}
\newtheorem{proposition}[theorem]{Proposition}
\newtheorem{remark}[theorem]{Remark}
\newtheorem{lemma}[theorem]{Lemma}

\newtheorem{definition}[theorem]{Definition}

\DeclareMathOperator{\supp}{supp}
\DeclareMathOperator{\card}{card}

\begin{document}

\title{Quantization of Gaussian measures with R\'enyi-$\alpha$-entropy constraints}
\author{Wolfgang Kreitmeier 
\footnote{\noindent 
Department of Informatics and Mathematics, \newline
University of Passau, 94032 Passau, Germany \newline e-mail: 
wolfgang.kreitmeier@uni-passau.de \newline
phone:  \texttt{+49(0)851/509-3014} \newline
This research was supported 
by a grant from the German Research Foundation (DFG). 
} 
} 

\date{}
\maketitle

\begin{abstract}
We consider the optimal quantization problem 
with R\'enyi-$\alpha$-entropy constraints 
for centered Gaussian measures on a separable Banach space.
For $\alpha = \infty$ we
can compute the optimal quantization
error by a moment on a ball.
For $\alpha \in {} ]1,\infty]$ and large entropy bound we derive sharp asymptotics for the
optimal quantization error in terms of the 
small ball probability of the Gaussian measure.
We apply our results to several classes of Gaussian measures.
The asymptotical order of the optimal quantization error for $\alpha > 1$ is different from
the well-known cases $\alpha = 0$ and $\alpha = 1$. 
\end{abstract}

~

\pagestyle{myheadings}
\markboth{W. Kreitmeier}{Quantization of Gaussian measures}

\section{Introduction and basic notation}

Let $\mathbb{N} := \{1,2,.. \}$. Let $\alpha \in \thinspace [0, \infty ]$ and 
$p=(p_{1},p_{2},...) \in [0,1]^{\mathbb{N}}$ be a probability vector,
i.e. $\sum_{i=1}^{\infty} p_{i} = 1$. 
The R\'enyi-$\alpha$-entropy 
$\hat{H}^{\alpha}(p) \in {} [0,\infty ]$
is defined as (see e.g. \cite[Definition 5.2.35]{ref_bib_Aczel} resp. \cite[Chapter 1.2.1]{ref_bib_Behara})  

\begin{equation}
\label{renyi_entr}
\hat{H}^{\alpha}(p)=\left\{ 
\begin{tabular}{l}
$- \sum_{i=1}^{\infty } p_{i} \log (p_{i}),$ if $\alpha = 1$ \smallskip \\ 
$- \log \left( \max \{ p_{i} : i \in \mathbb{N} \} \right)$, if $\alpha = \infty$ \smallskip \\
$\frac{1}{1-\alpha} \log \left(  \sum_{i=1}^{\infty } p_{i}^{\alpha} \right),$ 
if $\alpha \in {} [0,\infty[ \backslash \{1\}$.
\end{tabular}
\right.
\nonumber
\end{equation}
~\\
We use the convention $0 \cdot \log (0) :=0$ and $0^{x}:=0$ for all real $x$. 
The logarithm $\log$ is based on $e$.
\begin{remark}
\label{ref_rema_hospital}
With these conventions we obtain 
\[
\hat{H}^{0}(p) = \log \left(  \card \{ p_{i} : i \in \mathbb{N}, p_{i} > 0 \} \right),
\]
if $\card$ denotes cardinality.
Using the rule of de l'Hospital it is easy to see, that 
\[
\lim_{\alpha \rightarrow 1 \atop \alpha \neq 1 } \hat{H}^{\alpha}( \cdot ) = \hat{H}^{1}( \cdot )
\]
(cf. \cite[Remark 5.2.34]{ref_bib_Aczel}).
Moreover, $\lim_{\alpha \rightarrow \infty } \hat{H}^{\alpha}( \cdot ) = \hat{H}^{\infty}( \cdot )$.
\end{remark}

Let $(E, \| \cdot \|)$ be a real separable 
Banach space with norm $\| \cdot \| $.
Let $\mu$ be a Borel probability measure on $E$.
Denote by $\mathcal{F}$ the set of all
Borel-measurable mappings $f : E \rightarrow E$ with
$\card(f(E)) \leq \card (\mathbb{N})$.
A mapping $f \in \mathcal{F}$ is called \textbf{quantizer} and the image 
$f(E)$ is called codebook consisting of codepoints.
We assume throughout the whole paper that the codepoints are distinct. 
Every quantizer $f$ induces a partition
$\{ f^{-1}(z) : z \in f(E)  \}$ of $E$.
Every element of this partition is called \textbf{codecell}.
The image measure $\mu \circ f^{-1}$ has a countable support and defines an approximation of $\mu$, 
the so-called quantization of $\mu$ by $f$. 
For any enumeration $\{ z_{1},z_{2},.. \}$ of $f(E)$ we define
\[
H^{\alpha}_{\mu }(f) = \hat{H}^{\alpha } ( ( \mu \circ f^{-1} ( z_{1} ), \mu \circ f^{-1} ( z_{2} ),...) )
\]
as the R\'enyi-$\alpha$-entropy of $f$ w.r.t $\mu$.
Now we intend to quantify the distance between $\mu$ and its approximation under $f$.
To this end let  
$\rho: [0,\infty[ \rightarrow [0,\infty[$ be a 
surjective, strictly increasing and continuous 
mapping with $\rho(0)=0$. Hence $\rho$ is invertible. The inverse function is denoted by $\rho^{-1}$
and also strictly increasing.
We assume throughout the whole paper that $\int \rho ( \| x \| ) d \mu (x) < \infty $.
For $f \in \mathcal{F}$ we define as distance between $\mu$ and $\mu \circ f^{-1}$ 
the quantization error 
\begin{equation}
\nonumber
\label{ref_dmuf}
D_{\mu , \rho }(f)= \int \rho ( \parallel x-f(x)\parallel ) d \mu (x).
\end{equation} 
For any $R \geq 0$ we denote by 
\begin{equation}
\label{ref_darst_f_q}
D_{\mu , \rho}^{\alpha }(R)= \inf \{ D_{\mu , \rho}(f) : f \in \mathcal{F}, H^{\alpha }_{\mu }(f) \leq R \}
\end{equation}
the optimal quantization error for $\mu$ under R\'enyi-$\alpha$-entropy bound $R$.
Indeed, it is justified to speak of a distance.
It was shown by the author \cite{ref_bib_kreit4} in the finite-dimensional case and for 
Euclidean norm that for a large class of distributions $\mu$
the optimal quantization error (\ref{ref_darst_f_q}) is equal to 
a Wasserstein distance. 

\begin{remark}
\label{ref_rem_opt_quant}
The optimal quantization error is decreasing in $\alpha \geq 0$. To
see this let $f \in \mathcal{F}$ with $H_{\mu}^{\alpha }(f) \leq R$. 
For arbitrary $0< \gamma \leq \beta < \infty$ we 
have (cf. \cite{ref_bib_beck}, p. 53)
\[
H_{\mu}^{\beta }(f) \leq H_{\mu}^{\gamma }(f).
\]
Together with Remark \ref{ref_rema_hospital} we conclude that 
$H_{\mu}^{\alpha }(f) \leq H_{\mu}^{0 }(f)$. In view of Definition (\ref{ref_darst_f_q})
we thus obtain 
\[
D_{\mu, \rho }^{\alpha }(R) \leq \inf \{ D_{\mu, \rho }(g) : 
g \in \mathcal{F}, H^{0}_{\mu }(g) \leq R \} = D_{\mu,  \rho }^{0}(R). 
\]
\end{remark}

An exact determination of the optimal quantization error  
(\ref{ref_darst_f_q}) for every $R \geq 0$ 
was successfully only in a few special cases so far.
In this regard most is known in the one-dimensional case 
under the restriction 
\begin{equation}
\label{def_dist_f}
\rho (x) = x^{r} \text{ with } r > 0
\end{equation}
and $\alpha \in \{0,1\}$. In case of $\alpha = 0$ 
the reader is referred to \cite[section 5.2]{ref_bib_Graf_Luschgy_1}.
To the author's knowledge the uniform and the exponential distribution  
are the only examples for $\alpha = 1$ where an exact determination of
the optimal quantization error was carried out so far.
Gy\"{o}rgy and Linder \cite{ref_bib_Gyo00} have determined a parametric representation
of (\ref{ref_darst_f_q}) for the uniform distribution and a large class of 
distance functions $\rho$ which includes (\ref{def_dist_f}).
Berger \cite{ref_bib_Ber72} has derived in case of $\alpha = 1$ and $r=2$ 
an analytical representation for the optimal quantization error of the exponential distribution.
For the class (\ref{def_dist_f}) of distance functions the author \cite{ref_bib_kreit2} was able to
generalize the results of Gy\"{o}rgy and Linder \cite{ref_bib_Gyo00} 
to the case $\alpha \in [0, \infty ]$.  

Due to the difficulties in determining 
the optimal quantization error 
one is interested in asymptotics for the error for large entropy bounds.
In case of $\alpha \in \{ 0, 1 \}$ and finite dimension the asymptotical behaviour of the
optimal quantization error is well-known for a large class of distributions, see e.g.
\cite{ref_bib_Graf_Luschgy_1, ref_bib_Gray}.
Kreitmeier and Linder \cite{ref_bib_kreit5} have derived also sharp asymptotics for a large
class of one-dimensional distributions and $\alpha \in [0 , \infty]$. Moreover, the author
\cite{ref_bib_kreit3} has determined first-order asymptotics for the optimal quantization error 
in arbitrary finite dimension and $\alpha \in [0, \infty]$, where the class of distributions
is larger than the one in \cite{ref_bib_kreit5}.

This paper aims to determine asymptotics for the optimal quantization error (\ref{ref_darst_f_q})
in the infinite dimensional case.
To this end we will assume for the rest of this paper that $(E, \| \cdot \|)$ is 
of infinite dimension. Moreover, we restrict ourselves to Gaussian measures. In more detail 
we will assume from now on that $\mu$ 
is a non-atomic centered Gaussian measure on $E$ and the support of $\mu$
coincides with $E$. 
The restriction to Gaussian measures is motivated by different reasons. 
First, this class of distributions has been extensively studied 
in the past. In the proofs of this paper we especially use concentration inequalities 
(cf. \cite{ref_bib_bogachev}) and small ball asymptotics 
(see e.g. \cite{ref_bib_bel, ref_bib_bronski, 
ref_bib_chen, ref_bib_fill, ref_bib_li2, ref_bib_mas, ref_bib_shao, ref_bib_tala}).
Secondly, for distance functions of type (\ref{def_dist_f}) and $\alpha \in \{0,1\}$
the asymptotical order of $D_{\mu , \rho}^{\alpha }(R)$
for large $R$ has been already determined for several classes of 
Gaussian measures. 
Dereich et al. \cite{ref_bib_der} have determined asymptotics for 
(\ref{ref_darst_f_q}) in case of $\alpha = 0$ and for distance functions of type
(\ref{def_dist_f}). Their results require weak conditions
on the regular variation of the small ball asymptotics of the Gaussian measure.
Graf, Luschgy and Pag\`es \cite{ref_bib_Graf_Luschgy_3} have additionally shown 
for $\alpha = 0$ and restriction (\ref{def_dist_f})
that one can determine the small ball asymptotics 
from the asymptotics of the optimal quantization error (\ref{ref_darst_f_q})
if the asymptotics of (\ref{ref_darst_f_q}) satisfy certain regularity conditions. 
Luschgy and Pag\`es \cite{ref_bib_luschg} have determined sharp error asymptotics for $\alpha = 0$
and distance function $\rho(x)=x^2$. They imposed a condition on the regularity of
the eigenvalues of the covariance operator of $\mu$. 
In this situation, also the sharp error asymptotics for $\alpha = 0$ and $\alpha = 1$
coincide, cf. \cite{ref_bib_Graf_Luschgy_4}. 
Dereich and Scheutzow \cite{ref_bib_dereich} have shown for fractional
Brownian motion that sharp asymptotics of (\ref{ref_darst_f_q}) for large $R$ exist   
and also coincide for $\alpha \in \{ 0,1 \}$. 
According to these cited works the asymptotics for $\alpha = 0$
and $\alpha = 1$ are of the same order and in view of 
Remark \ref{ref_rem_opt_quant} even for all $\alpha \in [0,1]$.

The objective of this paper is to analyze the optimal quantization error for $\alpha > 1$. 
In Section 2 we determine 
in case of $\alpha = \infty$ ('mass-constrained quantization') 
a representation of the optimal quantization
error by a moment on a ball (cf. Proposition \ref{ref_prop_ball_repres}).
The proof of this result is a straightforward generalization of the techniques 
used in the proof of \cite[Proposition 2.1.]{ref_bib_kreit3}. 
In Section 3, for a large class of Gaussian measures where sharp asymptotics for the small ball probability are
known, we can determine sharp asymptotics for the optimal quantization error with entropy parameter $\alpha >1$
(cf. Corollary \ref{ref_coro_mappsdfre}, Theorem \ref{ref_prop_asym_sh}).
The cornerstone of our approach is covered by Proposition \ref{prop_first_part}.
For distance functions of type (\ref{def_dist_f}) we 
obtain a representation of the sharp asymptotics for $D_{\mu , \rho}^{\infty }$
in terms of the inverse of the small ball function (cf. definition (\ref{sm_b_fct})).
The condition imposed (cf. (\ref{ref_assump})) on the small ball asymptotics is satisfied by most 
prevalent Gaussian measures. For those distributions we are then able to derive 
also sharp asymptotics for all $\alpha > 1$, cf. Corollary \ref{ref_coro_mappsdfrexx}.
In Section 4 we discuss several examples of Gaussian processes
in order to determine the asymptotical order of the optimal quantization error
for large entropy bound and $\alpha > 1$.
The asymptotics of the optimal quantization error for $\alpha > 1$ turns out to be 
of different order compared to
the case $\alpha \leq 1$.

\section{The optimal quantization error under mass-constraints}

Let $f \in \mathcal{F}$ and $R>0$ with $H_{\mu}^{\infty }(f) \leq R$. From the definition 
we obtain
\[
\max \{ \mu ( f^{-1}(a) ) : a \in f(E) \}  \geq e^{-R} .
\]
Hence we call optimal quantization with $\alpha = \infty$ mass-constrained quantization. 
Denote by $\mathbb{R}$ all real numbers,  
let 
\[
\mathbb{R}^{+} = \{ x \in \mathbb{R} : x > 0 \} \text{ and }
\mathbb{R}_{0}^{+} = \{ x \in \mathbb{R} : x \geq 0 \} .
\]
As a key tool we will use Anderson's inequality \cite{ref_bib_and} as stated in 
reference \cite{ref_bib_bogachev}.

\begin{theorem}[$\text{\cite[Corollary 4.2.3]{ref_bib_bogachev}}$]
\label{anderson}
If $A$ is a convex, symmetric and Borel-measurable subset of $E$, then for every $a \in E$
\[
\mu ( A ) \geq \mu (A + a) .
\] 
Moreover, the function 
\[
\mathbb{R} \ni t \to \int_{E} g(x + t a) \mu (dx) 
\]
is nondecreasing on $\mathbb{R}_{0}^{+}$,
provided $g : E \to \mathbb{R}$ is such that the sets \\
$\{ g \leq c \}, {} c \in \mathbb{R}$, are
symmetric and convex, and $g( \cdot + ta )$ is $\mu-$integrable for any $t \geq 0$.  
\end{theorem}
 
We denote by $\supp (\mu )$ the support of $\mu$. For $a \in E$ and $s>0$ we denote by
\[
B(a, s) = \{ x \in E : \| x - a \| \leq s \}
\]
the closed ball around $a$ with radius $s$.  
We deduce from \cite[Corollary 4.4.2 (i)]{ref_bib_bogachev} that
the mapping
\begin{equation}
\label{ref_equ_fa}
\mathbb{R}^{+} \ni t \stackrel {F_{a}} { \rightarrow } \mu ( B( a,t ) ) \in \mathbb{R}^{+}
\end{equation}
is continuous. Because $\mu$ is non-atomic the mapping $F_{a}$ has a continuous extension to 
$\mathbb{R}_{0}^{+}$ which we call also $F_{a}$ and $F_{a}(0)=0$. 
For any set $A \subset E$ we denote by $1_{A}$ the characteristic function on $A$.

\begin{lemma}
\label{ref_lemm_abs0}
Let $a \in E$ and $A \subset E$ be a Borel measurable set with
$\mu (A) \in {} ]0,1[$. Then there exists an $s \in {} ]0, \infty [$ such that $\mu (A) = \mu (B(0,s))$ and
\[
\int_{A} \rho ( \| x - a \| ) d \mu (x)  
\geq \int_{B(0,s)} \rho ( \| x  \| ) d \mu (x) . 
\]
\end{lemma}
\begin{proof} ~\\
1. $\int_{A} \rho ( \| x - a \| ) d \mu (x) \geq \int_{B(a,l)} \rho ( \| x - a \| ) d \mu (x)$ with
$\mu ( B(a,l) ) = \mu (A)$. \smallskip \\
By the properties of the mapping $F_{a}$ an $l > 0$ exists 
with $\mu ( B(a,l) ) = \mu (A)$.
The remaining part of the proof can be taken from the proof of \cite[Lemma 2.8]{ref_bib_Graf_Luschgy_1}. 
Although \cite[Lemma 2.8]{ref_bib_Graf_Luschgy_1} covers only the special case $\rho (x) = x^{r}$,
the argument works also for general $\rho$. \smallskip \\
2. $\int_{B(a,l)} \rho ( \| x - a \| ) d \mu (x) \geq \int_{B(0,s)} \rho ( \| x \| ) d \mu (x)$. \smallskip \\
Let
\[
E \ni x \rightarrow f(x) = 1_{B(0,s)}(x) \rho ( \| x \| ) .
\]
By Theorem \ref{anderson}
an $s>0$ exists such that 
\[
\mu ( B( 0,s ) ) = \mu ( B( a,l ) ) \leq \mu ( B( 0,l ) ),
\]
which yields $s \leq l$. For every $c \in [0, \infty [$ the set 
\[
B(0,s) \cap \{ f \leq c \} = \{ x \in B(0,s) : \rho ( \| x \| ) \leq c \}  =  B ( 0, \min ( s, \rho^{-1}(c) ) )
\]
is symmetric and convex. 
Moreover $f( \cdot + t (-a))$ is $\mu$-integrable 
for every $t \geq 0$. Because the support of $f$ is a subset of $B( 0,s )$ 
we obtain from Theorem \ref{anderson} that
\begin{eqnarray*}
\int_{B(a,l)} \rho ( \| x - a \| ) d \mu (x) & \geq &
\int_{B(a,s)} \rho ( \| x - a \| ) d \mu (x) \\
&=& \int f(x-a) d \mu (x) \geq \int f(x) d \mu (x)
\\ &=& \int_{B(0,s)} \rho ( \| x \| ) d \mu (x) .
\end{eqnarray*}
\end{proof}

The proof of the following Proposition \ref{ref_prop_ball_repres} is
an obvious generalization of the proof of \cite[Proposition 2.1.]{ref_bib_kreit3} 
to the finite-dimensional case.
The main idea of constructing a quantizer based on a countable partition of $E$ works
also for infinite dimensional separable Banach spaces.
For the reader's convenience we provide a complete proof.

\begin{proposition}
\label{ref_prop_ball_repres}
Let $R>0$ and $s>0$ such that $\mu ( B(0,s) ) = e^{-R}$. Then 
\begin{equation}
\label{ref_equ_darstll}
D_{\mu , \rho }^{\infty }(R) = \int_{B(0,s)} \rho ( \| x \| ) d \mu (x).
\end{equation}
\end{proposition}
\begin{proof}
Let $R>0$.
From the definition (\ref{ref_darst_f_q}) of $D_{\mu, \rho}^{\infty }(R)$ we obtain
\[
D_{\mu, \rho}^{\infty }(R) = \inf \{ \int \rho ( \| x - f(x) \| ) d \mu (x) : f \in \mathcal{F}, 
\max_{a \in f(E)} \mu ( f^{-1}(a) ) \geq e^{-R} \}
\]
Moreover let
\begin{equation}
\label{ref_dr_inta}
D(R) = \inf \{ \int_{A} \rho ( \| x-a \| ) d \mu (x) : a \in E, A \text{ measurable }, \mu (A) \geq e^{-R} \}.
\end{equation}
1. $D_{\mu, \rho}^{\infty }(R) \geq D(R)$. \smallskip \\
Let $f \in \mathcal{F}$ with $\max_{b \in f(E)} \mu ( f^{-1}(b) ) \geq e^{-R}$. Then an $a \in f(E)$ exists
with $\mu ( f^{-1}(a) ) \geq e^{-R}$. Let $A = f^{-1}(a)$. We obtain
\begin{eqnarray*}
\int \rho ( \| x - f(x) \| ) d \mu (x) &=& \sum_{b \in f(E)} \int_{f^{-1}(b)} \rho ( \| x - b \| ) d \mu (x) \\
& \geq & \int_{f^{-1}(a)} \rho ( \| x - a \| ) d \mu (x) \\
& = & \int_{A} \rho ( \| x - a \| ) d \mu (x) \geq D(R),
\end{eqnarray*}
which yields $D_{\mu, \rho}^{\infty }(R) \geq D(R)$. \smallskip \\
2. $D_{\mu, \rho}^{\infty }(R) \leq D(R)$. \smallskip \\
Let $A \subset E$ be measurable with $\mu (A) \geq e^{-R}$ and choose $a \in E$.
Let $\varepsilon > 0$. Because $\rho(0)=0$ and $\rho$ is continuous, a $\delta > 0$ exists such that 
for every $t \in [0, \delta ]$ we have $\rho(t) \leq \varepsilon$.
Let $(x_{n})_{n \in \mathbb{N}}$ be dense in $E$. Then $(B(x_{n}, \delta))_{n \in \mathbb{N}}$
is an open cover of $E$. Hence a Borel-measurable partition $(A_{n})_{n \in \mathbb{N}}$ of $E \backslash A$
exists, such that $A_{n} \subset B(x_{n}, \delta)$ for every $n \in \mathbb{N}$. Now we define 
the mapping $f: E \rightarrow E$ by  
\begin{equation}
f(x)=\left\{ 
\begin{tabular}{l}
$a$, if $x \in A$ \smallskip \\ 
$x_{n}$, if $x \in A_{n}$ .
\end{tabular}
\right.
\nonumber
\end{equation}
Obviously $f \in \mathcal{F}$ and $\max_{b \in f(E)} \mu ( f^{-1}(b) ) \geq \mu (A) \geq e^{-R}$.
We deduce
\begin{eqnarray*}
D_{\mu, \rho}^{\infty }(R) &\leq & \int_{E} \rho ( \| x - f(x) \| ) d \mu (x) \\
&=& \int_{A} \rho ( \| x - a \| ) d \mu (x) + \sum_{n \in \mathbb{N}} \int_{A_{n}} \rho ( \| x - x_{n} \| ) d \mu (x) \\
&\leq & \int_{A} \rho ( \| x - a \| ) d \mu (x) + \sum_{n \in \mathbb{N}} \varepsilon \mu (A_{n}) \\
&=& \int_{A} \rho ( \| x - a \| ) d \mu (x) + \varepsilon.
\end{eqnarray*}
Because $\varepsilon > 0$, $a \in E$ and the set $A \subset E$ were arbitrary we obtain that 
$D_{\mu, \rho}^{\alpha }(R) \leq D(R)$. \smallskip \\
3. Proof of equation (\ref{ref_equ_darstll}). \smallskip \\
From step 1 and 2 we deduce $D_{\mu, \rho}^{\infty }(R) = D(R)$. 
Obviously we can assume that $\mu (A) \in {} ]0,1[$ for the set $A$ in (\ref{ref_dr_inta}).
But then the assertion follows from Lemma \ref{ref_lemm_abs0}.
\end{proof}

For $s > 0$ let 
\begin{equation}
\label{sm_b_fct}
b_{\mu }(s) = - \log ( \mu ( \{ x \in E : \| x  \| \leq s \} ) ) = - \log ( \mu ( B(0,s) ) ) 
\end{equation}
be the small ball function of $\mu$.
Note that $b_{\mu }(\cdot)$ is continuous, surjective, strictly decreasing and, therefore, invertible
(see e.g. \cite[Lemma 2.3.5]{ref_bib_de}).
Thus we obtain as an immediate consequence of Proposition \ref{ref_prop_ball_repres} the following result.
\begin{corollary}
\label{ref_coro_imme}
Let $R>0$. Then $D_{\mu , \rho }^{\infty }(R) \leq \rho ( b_{\mu }^{-1}(R) ) e^{-R}$.
\end{corollary}

\section{High-rate error asymptotics}

In this section we will prove high-rate error asymptotics for the optimal quantization error
with entropy index $\alpha \in {} ]1, \infty ]$.
If the small ball function $b_{\mu}(\cdot)$ has a certain asymptotical behavior
we can determine the sharp asymptotics of the optimal quantization error for large
entropy bound (cf. Corollary \ref{ref_coro_mappsdfre}, 
Theorem \ref{ref_prop_asym_sh})
We begin with an upper bound for the optimal quantization error.
As with Proposition \ref{ref_prop_ball_repres} the proof of the following result
is a straightforward generalization of the proof of 
\cite[Proposition 2.1.]{ref_bib_kreit3}.

\begin{lemma}
\label{ref_lemm_conj22}
Let $\alpha > 1$ and $R>0$.
Then
\begin{eqnarray*}
D_{\mu, \rho }^{\alpha }(R) 
&\leq & 
\inf \{ \int_{B(a,s)} \rho ( \| x - a \| ) d \mu (x) : 
a \in E, s>0, \mu (B(a,s)) \geq  e^{-\frac{\alpha - 1}{\alpha }R}  \} \\
&=& D_{\mu  , \rho }^{\infty } \left( \frac{\alpha - 1}{\alpha }R \right).
\end{eqnarray*}
\end{lemma}
\begin{proof}
The second part of the assertion is an immediate consequence of Proposition \ref{ref_prop_ball_repres}
and Lemma \ref{ref_lemm_abs0}.
To prove the first part let 
$a \in E$ and $s>0$ with $\mu ( B( a, s ) ) \geq e^{-\frac{\alpha - 1}{\alpha }R}$.
Let $(a_{n})_{n \in \mathbb{N}}$ be a dense subset of $E$. Let $\varepsilon > 0$.
Because $\rho(0)=0$ and $\rho$ is continuous, a $\delta > 0$ exists such that $\rho (t) < \varepsilon$ for every
$t \in [0, \delta]$. Hence $(B(a_{n}, \delta ))_{n \in \mathbb{N}}$ is an open cover of $E$.
Thus a Borel-measurable partition $(A_{n})_{n \in \mathbb{N}}$ of $E \backslash B(a,s)$ exists, with
$A_{n} \subset B ( a_{n}, \delta )$ for every $n \in \mathbb{N}$. We define the mapping $f: E \rightarrow E$
by
\begin{equation}
f(x)=\left\{ 
\begin{tabular}{l}
$a$, if $x \in B(a, s)$ \smallskip \\ 
$a_{n}$, if $x \in A_{n}$ .
\end{tabular}
\right.
\nonumber
\end{equation}
Obviously $f \in \mathcal{F}$. Due to $\alpha > 1$ we obtain
\begin{eqnarray*}
H_{\mu }^{\alpha }(f) &=& \frac{1}{1 - \alpha } \log 
\left( \mu ( B( a, s ) )^{\alpha } + \sum_{n=1}^{\infty } \mu ( A_{n} )^{\alpha }  \right) \\
& \leq & \frac{1}{1 - \alpha } \log ( \mu ( B( a, s ) )^{\alpha } ) 
\leq  \frac{1}{1 - \alpha } \log ( e^{(1 - \alpha )R} ) = R. 
\end{eqnarray*}
As a consequence we get
\begin{eqnarray*}
D_{\mu , \rho }^{\alpha }(R) \leq D_{\mu, \rho }(f) &=& \int_{B(a, s)} \rho ( \| x - a \| ) d \mu (x)
+ \sum_{n=1}^{\infty } \int_{A_{n}} \rho ( \| x - a_{n} \| ) d \mu (x) \\
& \leq & \int_{B(a,s)} \rho ( \| x - a \| ) d \mu (x) + \sum_{n=1}^{\infty } \mu ( A_{n} ) \varepsilon \\
& \leq & \int_{B(a, s)} \rho ( \| x - a \| ) d \mu (x) + \varepsilon .  
\end{eqnarray*}
Because $\varepsilon > 0$ was chosen arbitrarily the assertion is proved.
\end{proof}

Let $r>0$.
From now on we will assume for the rest of this paper that
$\rho(x) = x^{r}$ for every $x \in \mathbb{R}_{0}^{+}$. 
To stress this choice for $\rho$ we 
write $D_{\mu , r}^{\alpha }(\cdot )$ instead of $D_{\mu , \rho}^{\alpha }(\cdot )$.

\begin{remark}
\label{ref_rem_fernique}
The $r-$th moment is always finite for Gaussian measures.
This can be deduced either from Fernique's 
theorem (cf. \cite[Theorem 2.8.5]{ref_bib_bogachev})
or follows from concentration inequalities 
for Gaussian measures (see e.g. \cite[Theorem 4.3.3]{ref_bib_bogachev} or \cite[p. 25]{ref_bib_de}).
\end{remark}

\noindent In order to formulate rates we introduce the following notations.
For mappings $f,g : \mathbb{R}_{0}^{+} \rightarrow \mathbb{R}^{+}$ 
and $a \in [0, \infty]$
we write $f \sim g$ as $x \rightarrow a$ if
$\lim_{x \rightarrow a} f(x)/g(x) = 1$. We denote $f \gtrapprox g$ as $x \rightarrow a$ if 
\[
0 < \liminf_{x \rightarrow a } f(x)/g(x)  
\]
and $f \gtrsim g$ as $x \rightarrow a$ if
\[
1 \leq \liminf_{x \rightarrow a } f(x)/g(x) . 
\]
We write $f \lessapprox g$ as $x \rightarrow a$ if
\[
\limsup_{x \rightarrow a } f(x)/g(x)  < \infty 
\] 
and $f \lesssim g$ as $x \rightarrow a$ if
\[
\limsup_{x \rightarrow a } f(x)/g(x)  \leq 1.
\] 
If $f \gtrapprox g$ and $f \lessapprox g$ we write $f \approx g$. 
Obviously $f \sim g$, if $f \gtrsim g$ and $f \lesssim g$. 
\smallskip \\
The following Lemma has been proved by Dereich (cf. \cite[Lemma 4.2]{ref_bib_der}).
See also the proof of Corollary 1.3 in \cite{ref_bib_Graf_Luschgy_3}.  

\begin{lemma}
\label{ref_lemm_invert}
Let $c>0$.
Let $a>0$ and $b \in \mathbb{R}$. 
If
\[
b_{\mu }(s) \sim c (1/s)^{a} \left( \log ( 1/s ) \right)^{b} \text{ as } s \rightarrow 0,
\]
then
\[
b_{\mu }^{-1}(R) \sim c^{1/a} a^{-b/a} R^{-1/a} \left( \log (R) \right)^{b/a} \text{ as } R \rightarrow \infty .
\]
\end{lemma}

\begin{remark}
\label{ref_rem_approx}
It is also easy to check that 
\[
b_{\mu }(s) \approx (1/s)^{a} \left( \log ( 1/s ) \right)^{b} \text{ as } s \rightarrow 0
\]
implies that
\[
b_{\mu }^{-1}(R) \approx R^{-1/a} \left( \log (R) \right)^{b/a} \text{ as } R \rightarrow \infty .
\]
\end{remark}

\begin{remark}
\label{ref_rem_separ}
According to the separability of the support of $\mu$ and the
finite $r-$th moment (cf. Remark \ref{ref_rem_fernique}) we have
\[
\lim_{R \rightarrow \infty} D_{\mu, r}^{0 }(R)=0.
\]
In view of Remark \ref{ref_rem_opt_quant} we thus get 
\[
\lim_{R \rightarrow \infty} D_{\mu, r}^{\alpha }(R)=0
\]
for every $\alpha \in [0, \infty]$.
\end{remark}

\begin{definition}
A family $(f_{R})_{R>0} \subset \mathcal{F}$ of quantizers is called 
asymptotically $\alpha - $optimal, if $H_{\mu}^{\alpha }(f_{R}) \leq R$ for
every $R>0$ and
\begin{equation}
\label{rhs_dasy}
\lim_{R \rightarrow \infty} \frac{D_{\mu , r}(f_{R})}{D_{\mu , r}^{\alpha }(R)} = 1.
\end{equation}
\end{definition}
\begin{remark}
\label{ref_Rem_fin}
An asymptotically $\alpha - $optimal family $(f_{R})_{R>0}$ always exists.  
Moreover we can assume w.l.o.g. that $\card ( f_{R}(E) ) < \infty$ for every 
$R > 0$. 
Let us justify this.
First, we note that $D_{\mu , r}^{\alpha}(R) > 0$ for every $R>0$, because
$\mu$ is a non-degenerate Gaussian measure. Hence, the left hand side 
of (\ref{rhs_dasy}) is well-defined.
Clearly, for every $R \geq 0$ we can define a quantizer $f \in \mathcal{F}$ 
with $H_\mu^{\alpha}(f)=0\leq R$
(choose $a \in E$ and let $f(x)=a$ for every $x \in E$).
Thus for every $R \geq 0$ a sequence $(f_{R}^{n})_{n \in \mathbb{N}}$
of quantizers exists with $H_{\mu}^{\alpha}(f_{R}^{n}) \leq R$ and 
$D_{\mu , r}(f_{R}^{n}) \to D_{\mu , r}^{\alpha}(R)$ as $n \to \infty$.
For every $R>0$ choose $n_{0}(R)$ such that
\[
|D_{\mu , r}(f_{R}^{n_{0}(R)}) - D_{\mu , r}^{\alpha}(R) | \leq 
\varepsilon_{R}
\text{ with }
\varepsilon_{R} = R^{-1} D_{\mu , r}^{\alpha}(R) .
\] 
Consequently, $(f_{R}^{n_{0}(R)})_{R>0}$ is an asymptotically $\alpha - $optimal family, i.e.
such a family always exists.

Now let $(f_{R})_{R>0}$ be
an asymptotically $\alpha - $optimal family, let $\varepsilon_{R}$ 
as above and $a_{0} \in f_{R}(E)$. Choose 
$A \subset f_{R}(E) \backslash \{ a_{0} \}$ such that $\card(f_{R}(E) \backslash A) < \infty$,
\[
\sum_{a \in A} \int_{f_{R}^{-1}(a)} \| x - a \|^{r} d \mu (x) \leq \varepsilon_{R} / 2
\text{ and }
\int_{\cup_{a \in A}f_{R}^{-1}(a)} \| x - a_{0} \|^{r} d \mu (x) \leq \varepsilon_{R} / 2.
\]
With the quantizer 
\[
g_{R} = a_{0} 1_{\cup_{a \in \{ a_{0} \}\cup A } f_{R}^{-1}(a)} ( \cdot ) + 
\sum_{a \in f_{R}(E)\backslash ( \{ a_{0} \}\cup A )} a 1_{f_{R}^{-1}(a)}(\cdot) 
\]
we obtain 
\[
|D_{\mu , r}(f_{R}) - D_{\mu , r}(g_{R}) | \leq \varepsilon_{R}.
\]
Moreover, $H_{\mu}^{\alpha}(g_{R}) \leq H_{\mu}^{\alpha}(f_{R})$
according to \cite[Proposition 4.2]{ref_bib_kreit4} 
and \cite[Definition 2.1(b)]{ref_bib_kreit4}.
Thus, $(g_{R})_{R>0}$ is also asymptotically $\alpha - $optimal with
$\card ( g_{R}(E) ) < \infty$ for every $R > 0$. 
\end{remark}

\begin{lemma}
\label{ref_lemm_xyz}
Let $\alpha \in {} ]1, \infty [$ and $(f_{R})_{R>0} \subset \mathcal{F}$ be an 
asymptotically $\alpha - $optimal family. Then,
$\lim_{R \rightarrow \infty}H_{\mu}^{\alpha }(R) = \infty$, or, what is the same,
\[
\lim_{R \rightarrow \infty } ( \sum_{a \in f_{R}(E)} \mu ( f_{R}^{-1}(a) )^{\alpha })^{1/\alpha } = 0.
\]
\end{lemma}
\begin{proof}
Let $C>0$ and $(R_{n})_{n \in \mathbb{N}}$ be a sequence with $\lim_{n \rightarrow \infty}R_{n} = \infty$.
We will show that $\liminf_{n \rightarrow \infty } H_{\mu}^{\alpha }(f_{R_{n}}) \geq C$.
Let us assume the contrary. 
Hence, there exists a subsequence of $(f_{R_{n}})_{n \in \mathbb{N}}$, which we will also denote by 
$(f_{R_{n}})_{n \in \mathbb{N}}$, such that
$H_{\mu}^{\alpha }(f_{R_{n}}) < C$ for every $n \in \mathbb{N}$.
By definition (\ref{ref_darst_f_q}) of the optimal quantization error and because the support of $\mu$ is infinite 
we obtain
\begin{equation}
\label{ref_equ_sdftr}
\liminf_{n \rightarrow \infty } D_{\mu , r}(f_{R_{n}}) \geq D_{\mu , r}^{\alpha }(C ) > 0.
\end{equation}
Let $\varepsilon > 0$. Then there exists an $R_{\varepsilon} > 0$, such that
\begin{equation}
\label{ref_lim_aid}
D_{\mu , r}(f_{R}) \leq ( 1 + \varepsilon ) D_{\mu , r}^{\alpha }(R) \quad \text{ for every }
R \geq R_{\varepsilon } .
\end{equation}
From (\ref{ref_lim_aid}) and Remark \ref{ref_rem_separ} we get
\[
0 \leq \lim_{R \rightarrow \infty }D_{\mu , r}(f_{R}) \leq  
( 1 + \varepsilon ) \lim_{R \rightarrow \infty }D_{\mu , r}^{\alpha  }(R) = 0,
\]
which contradicts (\ref{ref_equ_sdftr}). Thus we obtain that 
$\liminf_{n \rightarrow \infty } H_{\mu}^{\alpha }(f_{R_{n}}) \geq C$.
Because $C>0$ and $(R_{n})_{n \in \mathbb{N}}$ was arbitrary we get
$\lim_{R \rightarrow \infty} H_{\mu}^{\alpha }(f_{R}) = \infty$. 
Now the assertion follows immediately from the definition of $H_{\mu}^{\alpha }(f_{R})$.  
\end{proof}

\begin{remark}
\label{ref_rem_dfkjlkasdf}
As an immediate consequence of Lemma \ref{ref_lemm_xyz} we obtain
\[
\lim_{R \rightarrow \infty } \max_{a \in f_{R}(E)} \mu ( f_{R}^{-1}(a) ) = 0.
\]
for $\alpha \in {} ]1, \infty [$ and every
asymptotically $\alpha - $optimal family $(f_{R})_{R>0} \subset \mathcal{F}$. 
\end{remark}

\begin{proposition}
\label{prop_first_part}
If
\begin{equation}
\label{ref_assump}
\lim_{s \rightarrow 0} \frac{\mu ( B( 0, \eta s ) )}{\mu ( B( 0, s ) )} = 0
\end{equation}
for every $\eta \in {} ]0,1[$, then
\[
D_{\mu, r}^{\infty }(R) \sim (b_{\mu }^{-1}(R))^{r} e^{-R} \text{ as } R \rightarrow \infty .
\]
\end{proposition}
\begin{proof}
Let $\eta \in {} ]0,1[$ and $s \in {} ]0, \infty [$ with $s=b_{\mu }^{-1}(R)$, i.e. $\mu (B(0,s)) = e^{-R}$.
Proposition \ref{ref_prop_ball_repres} implies
\begin{eqnarray*}
\frac{D_{\mu , r}^{\infty }(R)}{(b_{\mu }^{-1}(R))^{r}e^{-R}} &=&
\frac{\int_{B(0,s)}\| x \|^{r} d \mu (x)}{\mu ( B(0, s) ) s^{r}} \geq 
\frac{1}{\mu ( B( 0,s ) ) s^{r}} \int_{B(0,s) \backslash B(0, \eta s)} \| x \|^{r} d \mu (x) \\
& \geq  & \eta ^{r} ( 1 - \frac{\mu ( B( 0, \eta s ) )}{\mu ( B( 0, s ) )}).
\end{eqnarray*}
Because $\eta \in {} ]0,1[$ is arbitrary and by assumption (\ref{ref_assump}) we obtain that
\begin{equation}
\label{ref_equ_wwerf}
\liminf_{R \rightarrow \infty} \frac{D_{\mu , r}^{\infty }(R)}{(b_{\mu }^{-1}(R))^{r}e^{-R}} \geq 1.
\end{equation}
From Corollary \ref{ref_coro_imme} we get
\[
\limsup_{R \rightarrow \infty} \frac{D_{\mu , r}^{\infty }(R)}{(b_{\mu }^{-1}(R))^{r}e^{-R}} \leq 1,
\]
which yields together with (\ref{ref_equ_wwerf}) the assertion.
\end{proof}

In the sequel we will apply results from the theory of 
slowly varying functions (cf. \cite[Definition 1.2.1]{ref_bib_bingham}).
Let us state first the exact definition of this notion. 
\begin{definition}
Let $z>0$ and $f: [z,\infty [ \to \mathbb{R}^{+}$ be a Borel-measurable 
mapping satisfying
\[
f( \lambda x  ) / f(x) \to 1 \text{ as } x \to \infty  \hspace{1.5em} \forall \lambda > 0; 
\]
then $f$ is said to be slowly varying.
\end{definition} 

\begin{corollary}
\label{ref_coro_mappsdfre}
Let $c>0$. Let $a>0$ and $b \in \mathbb{R}$. If
\[
b_{\mu }(s) \sim c (1/s)^{a} ( \log (1/s) )^{b} \text{ as } s \rightarrow 0,
\]
then
\[
D_{\mu , r}^{\infty } (R) \sim ( b_{\mu }^{-1}(R) )^{r} e^{-R} 
\sim c^{1/a}a^{-b/a} R^{-1/a} ( \log (R) )^{b/a} e^{-R}
\text{ as } R \rightarrow \infty .
\]
\end{corollary}
\begin{proof}
From \cite[p.16]{ref_bib_bingham} we obtain that $[1, \infty [ {} \ni x \rightarrow c ( \log (x) )^{b} $
is a slowly varying function. Let $\eta \in {} ]0,1[$.
Applying \cite[Theorem 1.4.1]{ref_bib_bingham}
we deduce 
from \cite[Definition (1.4.3)]{ref_bib_bingham}  
that 
\[
1 < \eta ^{-a} =
\lim_{x \rightarrow \infty } \frac{b_{\mu }( 1/ ( \eta^{-1} x ) )}{b_{\mu } (1/x)}
= \lim_{s \rightarrow 0} \frac{b_{\mu }(\eta s)}{b_{\mu }(s)} .
\]
Thus we obtain 
\[
\lim_{s \rightarrow 0} \frac{\mu ( B( 0, \eta s ) )}{\mu ( B( 0, s ) )} =
\lim_{s \rightarrow 0} e^{b_{\mu }(s)(1 - b_{\mu }( \eta s )/b_{\mu }(s))} = 0.
\]
Hence the first part of the assertion follows from Proposition \ref{prop_first_part}.
The second part is a direct consequence of Lemma \ref{ref_lemm_invert}.
\end{proof}

Before we can state and prove our main result (Theorem \ref{ref_prop_asym_sh})
we need two more technical lemmas.

\begin{lemma}
\label{ref_lemm_cont_diff}
Let $\alpha > 1$, $A>0$ and $B \geq 0$. Let $f: [0, 1/e [ {} \rightarrow \mathbb{R} $
with 
\begin{equation}
f(x)=\left\{ 
\begin{tabular}{l}
$0$, if $x=0$ \smallskip \\ 
$x \left( \log ( 1/x ) \right)^{-A} \left( \log ( \log ( 1/x ) ) \right)^{B}  ,$ if $x \in {} ]0,1/e[$.
\end{tabular}
\right.
\nonumber
\end{equation}
Then, $f$ is continuous on $[0, 1/e [$, continuously differentiable on $]0, 1/e [$ and
monotone increasing on $[0, e^{-e} [$.
Moreover an $x_{0} \in {} ]0,1/e[$ exists, such that the
mapping $]0, x_{0}[ {} \ni x \rightarrow F(x)=x^{1 - \alpha} f^{\prime}(x) $ is monotone decreasing.
\end{lemma}
\begin{proof}
Let $z>1$ and $g(z)=z f(1/z)$.
Thus $g$ is slowly varying (cf. \cite[Examples p. 16]{ref_bib_bingham}). 
Applying \cite[Proposition 1.3.6 (v)]{ref_bib_bingham} we obtain
\[
\lim _{x \rightarrow 0} f(x) =
\lim _{z \rightarrow \infty } (1/z) g(z) = 0.
\]
Thus $f$ is continuous on $[0, 1[$.
Now let $x \in {} ]0, 1[$ and $z=1/x$. We calculate
\[
x^{1-\alpha } f^{\prime }(x)= z^{\alpha - 1} h_{1}(z) \cdot h_{2}(z)
\]
with
\[
h_{1}(z) = ( \log ( \log (z) ) )^{B-1} ( \log (z) )^{-A-1}
\]
and
\[
h_{2}(z) = \log ( \log ( z ) ) \log ( z ) + \log ( \log ( z ) ) - 1.
\]
Because $f^{\prime}(x)=h_{1}(1/x)h_{2}(1/x)>0$ for every $x \in {} ]0, e^{-e}[$ we obtain
that $f$ is monotone increasing on $[0,e^{-e}[$.
Obviously $h_{1}$ and $h_{2}-1$ are slowly varying (cf. \cite[p. 16]{ref_bib_bingham}).
Due to $h_{2}(z) \rightarrow \infty$ as $z \rightarrow \infty$ we obtain that also
$h_{2}$ is slowly varying (cf. \cite[Definition p. 6]{ref_bib_bingham}) and,
therefore, by means of \cite[Proposition 1.3.6 (iii)]{ref_bib_bingham}, that
$h_{1} \cdot h_{2}$ is slowly varying.
From \cite[Theorem 1.3.1]{ref_bib_bingham} we deduce that 
$h_{1} \cdot h_{2}$ is a normalized slowly varying function (cf. \cite[Definition p. 15]{ref_bib_bingham}).
Now \cite[Theorem 1.5.5]{ref_bib_bingham} yields that $h_{1} \cdot h_{2}$ is an element of the so-called 
Zygmund class. From the definition of the Zygmund class (cf. \cite[Definition p. 24]{ref_bib_bingham})
we deduce that a $z_{0}>0$ exists, such that
$z \rightarrow z^{\alpha - 1} h_{1}(z) \cdot h_{2}(z)$ is increasing on $[z_{0}, \infty[$.
But then $F$ is decreasing on $]0,x_{0}[$ with $x_{0}=1/z_{0}$, which finally
proves the assertion.
\end{proof}

\begin{lemma}
\label{ref_lemm_low_bound}
Let $\alpha \in {} ]1, \infty [$.
Let $x_{0} > 0$ and
$f : [0, x_{0}] {} \rightarrow \mathbb{R}_{0}^{+}$ be continuous with $f(0)=0$. 
If $f$ is continuously differentiable on $]0, x_{0}[$ and the
mapping $]0, x_{0}[ {} \ni x \rightarrow F(x)=x^{1 - \alpha} f^{\prime}(x) $ is monotone decreasing, then
\[
f(x_{0}) \leq  \inf \left\{ \sum_{i=1}^{n} f(x_{i}) : n \in \mathbb{N} , 
(x_{1},..,x_{n}) \in {} [0,x_{0}]^{n}, \sum_{i=1}^{n}x_{i}=1, 
\sum_{i=1}^{n} x_{i}^{\alpha } \geq x_{0}^{\alpha }  \right \} .
\]  
\end{lemma}
\begin{proof}
For $x \in [0, x_{0}^{\alpha }]$ let $g (x) = f (x^{1 / \alpha})$. 
For $n \in \mathbb{N}$ let
\[
A_{n, \alpha }(x_{0}) = \{ (x_{1},..,x_{n}) \in {} [0,x_{0}^{\alpha }]^{n} : 
\sum_{i=1}^{n} x_{i} \geq x_{0}^{\alpha } \} .
\]
Thus we obtain
\begin{eqnarray*}
&& \inf \left\{ \sum_{i=1}^{n} f(x_{i}) : n \in \mathbb{N} , 
(x_{1},..,x_{n}) \in {} [0,x_{0}]^{n}, \sum_{i=1}^{n}x_{i}=1,
\sum_{i=1}^{n} x_{i}^{\alpha } \geq x_{0}^{\alpha }  \right \} \\
&\geq &
\inf \left\{ \sum_{i=1}^{n} g(x_{i}) : n \in \mathbb{N} , 
(x_{1},..,x_{n}) \in {} A_{n, \alpha }(x_{0})  \right \}.
\end{eqnarray*}
$A_{n, \alpha }(x_{0})$ is a
convex set for every $n \in \mathbb{N}$.
Moreover $g$ is differentiable on $]0, x_{0}^{\alpha}[$ and 
for $x \in {} ]0, x_{0}[$ we calculate 
$g^{\prime }(x^{\alpha })= \alpha ^{-1} x^{1- \alpha } f^{\prime }(x).$
Thus $g^{\prime }$ is monotone decreasing on $]0, x_{0}^{\alpha}[$ and, therefore, concave on $]0, x_{0}^{\alpha}[$.
For $n \in \mathbb{N}$ and $x \in A_{n, \alpha }(x_{0})$
let $G(x)=\sum_{i=1}^{n} g(x_{i})$. Obviously $G$ is continuous and concave on 
the convex compact set $A_{n, \alpha }(x_{0})$.
Applying \cite[Theorem 3.4.7.]{ref_bib_nicu} 
we obtain that $G$ attains its global minimum at an extreme point of $A_{n, \alpha }(x_{0})$.
Note that the extreme points of $A_{n, \alpha }(x_{0})$ are consisting of the set
\[
\{ (x_{1}, ..,x_{n}) \in [ 0,x_{0}^{\alpha } ]^{n} : 
x_{i} \in \{ 0, x_{0}^{\alpha } \} \text{ for every } i \in \{ 1,..,n \} \} \backslash \{ 0 \}.
\]
Thus we get
\[
\inf \left\{ \sum_{i=1}^{n} g(x_{i}) : n \in \mathbb{N} , 
(x_{1},..,x_{n}) \in {} A_{n, \alpha }(x_{0})  \right \}
\geq g(x_{0}^{\alpha }) = f(x_{0}),
\]
which yields the assertion.
\end{proof}

\begin{theorem}
\label{ref_prop_asym_sh}
Let $\alpha \in {} ]1,\infty[$ and
$c>0$. Let $a>0$ and $b \in \mathbb{R}$. If
\[
b_{\mu }(s) \sim c (1/s)^{a} ( \log (1/s) )^{b} \text{ as } s \rightarrow 0,
\]
then
\[
D_{\mu, r }^{\alpha } (R) \sim 
\left( b_{\mu }^{-1}
\left( \frac{\alpha - 1}{\alpha } R \right) \right)^{r} e^{- \frac{\alpha - 1}{\alpha }R}
\sim D_{\mu  , r }^{\infty } \left( \frac{\alpha - 1}{\alpha }R \right)
\]
as $R \rightarrow \infty$.
\end{theorem}
\begin{proof}
1. $D_{\mu, r}^{\alpha }(R) \lesssim \left( b_{\mu }^{-1}
\left( \frac{\alpha - 1}{\alpha } R \right) \right)^{r} e^{- \frac{\alpha - 1}{\alpha }R}
$ as $R \rightarrow \infty$. \smallskip  \\
This follows from Lemma \ref{ref_lemm_conj22} and Corollary \ref{ref_coro_mappsdfre}. \medskip \\
2. $D_{\mu, r}^{\alpha }(R) \gtrsim \left( b_{\mu }^{-1}
\left( \frac{\alpha - 1}{\alpha } R \right) \right)^{r} e^{- \frac{\alpha - 1}{\alpha }R}$ 
as $R \rightarrow \infty$. \smallskip  \\
Let $(f_{R})_{R > 0}$ be an asymptotically $\alpha$-optimal family of quantizers.
According to Remark \ref{ref_Rem_fin} let us assume w.l.o.g. that $\card (f_{R}(E)) < \infty$ for every $R>0$.
By definition we have 
\[
D_{\mu , r}(f_{R}) = \sum_{a \in f_{R}(E)} \int_{f_{R}^{-1}(a)} \| x - a \|^{r} d \mu (x) .
\]
For every $a \in f_{R}(E)$ let $s_{a}(R)>0$ such that $\mu (B(0, s_{a}))=\mu (f_{R}^{-1}(a))$.
Applying Lemma \ref{ref_lemm_abs0} and Proposition \ref{ref_prop_ball_repres} we deduce
\[
\int_{f_{R}^{-1}(a)} \| x - a \|^{r} d \mu (x) \geq
\int_{B(0, s_{a})} \| x \|^{r} d \mu (x) = D_{\mu , r}^{\infty }(- \log ( \mu ( f_{R}^{-1}(a) ) )).
\]
Now let $\varepsilon \in {} ]0,1[$. According to Corollary \ref{ref_coro_mappsdfre} a $\delta \in {} ]0,1[$ exists,
such that for every $x \in {} ]0, \delta[$ 
\[
D_{\mu , r}^{\infty } (- \log (x)) \geq ( 1 - \varepsilon ) C \cdot g(x)
\]
with $C=c^{1/a}a^{-b/a}$ and $g(x)=x(-\log(x))^{-1/a}( \log ( - \log ( x ) ) )^{b/a}$. 
From Remark \ref{ref_rem_dfkjlkasdf} we get 
an $R_{1} > 0$ such that 
for every $R \geq R_{1}$ and for every $a \in f_{R}(E)$ we have
\[
D_{\mu , r}^{\infty }( - \log ( \mu ( f_{R}^{-1}(a) ) ) ) 
\geq ( 1 - \varepsilon ) C \cdot g (\mu ( f_{R}^{-1}(a) ))
\]
and, therefore,
\begin{equation}
\label{ref_equ_4535wefr}
D_{\mu , r}(f_{R}) \geq ( 1 - \varepsilon ) C \sum_{a \in f_{R}(E)} g(\mu ( f_{R}^{-1}(a) )).
\end{equation}
Applying Lemma \ref{ref_lemm_cont_diff} we obtain a $z_{0} \in {} ]0, \delta[$ such that
$g$ is monotone increasing on $[0, z_{0}[$ and 
the mapping $]0,z_{0}[ {} \ni x \rightarrow x^{1-\alpha } g^{\prime}(x)$
is monotone decreasing. 
Let $x_{0}=x_{0}(R, f_{R})=( \sum_{a \in f_{R}(E)} \mu ( f_{R}^{-1}(a) )^{\alpha } ) ^{1/\alpha }$
and choose according to Lemma \ref{ref_lemm_xyz} an $R_{2} > 0$ such that 
$x_{0} < z_{0}$ for every $R \geq R_{2}$.
Now let $R \geq \max(R_{1}, R_{2})$.
From $H_{\mu }^{\alpha }(f_{R}) \leq R$ we obtain $z_{0} > x_{0} \geq e^{-\frac{\alpha - 1}{\alpha }R}$.
Hence we can apply Lemma \ref{ref_lemm_low_bound} 
and deduce together with the monotonicity of $g$ that
\begin{equation}
\label{ref_equ_adsoiuiouerw}
\sum_{a \in f_{R}(E)} g ( \mu ( f_{R}^{-1}(a) ) ) \geq g ( x_{0} ) \geq g( e^{-\frac{\alpha - 1}{\alpha }R} ).
\end{equation}
Combining (\ref{ref_equ_4535wefr}) and (\ref{ref_equ_adsoiuiouerw}) 
we obtain from Corollary \ref{ref_coro_mappsdfre} that 
\[
1 - \varepsilon \leq
\liminf_{R \rightarrow \infty } \frac{D_{\mu , r}(f_{R})}{C g( e^{-\frac{\alpha - 1}{\alpha }R} )}
= \liminf_{R \rightarrow \infty } \frac{D_{\mu, r}^{\alpha }(R)}{\left( b_{\mu }^{-1}
\left( \frac{\alpha - 1}{\alpha } R \right) \right)^{r} e^{- \frac{\alpha - 1}{\alpha }R}}.
\]
Because $\varepsilon$ was arbitrary this proves the assertion of step 2. 
\smallskip \\
3. $D_{\mu  , r }^{\infty } \left( \frac{\alpha - 1}{\alpha }R \right)
\sim \left( b_{\mu }^{-1}\left( \frac{\alpha - 1}{\alpha } R \right) \right)^{r} e^{- \frac{\alpha - 1}{\alpha }R}$
as $R \rightarrow \infty$. \smallskip \\
This follows from Corollary \ref{ref_coro_mappsdfre}.
\end{proof}

\begin{corollary}
\label{ref_coro_mappsdfrexx}
Let $\alpha \in {} ]1,\infty[$ and
$c>0$. Let $a>0$ and $b \in \mathbb{R}$. If
\[
b_{\mu }(s) \sim c (1/s)^{a} ( \log (1/s) )^{b} \text{ as } s \rightarrow 0,
\]
then
\[
D_{\mu, r }^{\alpha } (R) \sim 
c^{1/a}a^{-b/a} (\frac{\alpha - 1}{\alpha } R)^{-1/a} ( \log (\frac{\alpha - 1}{\alpha } R) )^{b/a} e^{-\frac{\alpha - 1}{\alpha } R}
\]
as $R \rightarrow \infty$.
\end{corollary}
\begin{proof}
Immediate consequence of Theorem \ref{ref_prop_asym_sh} and Corollary \ref{ref_coro_mappsdfre}.
\end{proof}

\begin{remark}
Unfortunately the author was unable to prove the assertion of 
Theorem \ref{ref_prop_asym_sh} under the weaker condition (\ref{ref_assump}) of Proposition \ref{prop_first_part}. 
In general it would be of interest to characterize those Gaussian measures who are 
satisfying condition (\ref{ref_assump}).
Finally it is also an open question if the asymptotics in Theorem \ref{ref_prop_asym_sh}
resp. Corollary \ref{ref_coro_mappsdfre} remain valid in a weak sense if
we only require weak asymptotics for the small ball function, i.e. if
we replace $\sim$ by $\approx$.
\end{remark}

\section{Examples}

Let $d \in \mathbb{N}$.
Let $f : [0,1]^{d} \rightarrow \mathbb{R}$. 
We denote by $\| \cdot \|_{\infty }$ the sup-norm, i.e. 
\[
\| f \|_{\infty} = \sup_{t \in [0,1]^{d}} | f(t) | .
\]
Moreover, for $p \geq 1$ and $p-$integrable mapping $f$ let 
\[
\| f \|_{p} = \left( \int_{[0,1]^{d}} | f |^{p} dt \right)^{1/p}
\]
be the $L_{p}-$norm of $f$.
In the sequel we consider
centered Gaussian probability vectors $X=(X_{t})_{t \in [0,1]^{d}}$
on the separable Banach space $(C, \| \cdot \|_{\infty})$ of
all continuous functions on $[0,1]^{d}$ and on the separable Banach space $(L_{p}, \| \cdot \|_{p})$
of all $p-$integrable functions, respectively. 

To be more precise we write $D_{\mu, r,  \| \cdot \|_{\infty }}^{\alpha }( \cdot )$ resp.
$D_{\mu, r,  \| \cdot \|_{ p }}^{\alpha }( \cdot )$ 
for the optimal quantization error in order to stress the dependency on the underlying norm
and norm exponent $r>0$. 
Moreover we write $b_{\mu, \| \cdot \|_{\infty }}(s)$ resp. $b_{\mu, \| \cdot \|_{p}}(s)$ for the 
small ball function.

Although results for small ball probabilities of Gaussian measures are also available for
other norms (see e.g. \cite{ref_bib_lifs}) we restrict ourself to the two norms from above. 

\subsection{Fractional Brownian sheet}

Let $H = (H_{1},..,H_{d}) \in {} ]0,1[^{d}$. 
Consider the centered Gaussian probability vector 
\[
X^{H}=(X_{t}^{H})_{t \in [0,1]^{d}}
\]
characterized by the covariance function
\[
\mathbb{E} X_{t}^{H} X_{s}^{H} = \prod_{i=1}^{d} \frac{ s_{i}^{2H_{i}} + t_{i}^{2H_{i}} - | s_{i}-t_{i} |^{2H_{i}} }{2},
\] 
with $s,t \in [0,1]^{d}$.
Fractional Brownian motion is covered by the special case $d=1$.
Moreover we obtain the classical Brownian sheet by letting $d=2$ and $H_{1}=H_{2}=1/2$.
Let $\gamma = \min ( H_{1},..,H_{d} ) > 0$.
\smallskip \\
Case 1. there is a unique minimum among $H_{1},..,H_{d}$. \\ 
In this case we know, that 
a $c=c(H) \in {} ]0, \infty[$ exists, such that 
\[
b_{\mu, \| \cdot \|_{\infty }}(s) \sim c s^{-1/\gamma } \text{ as } s \rightarrow 0
\]
(cf. \cite{ref_bib_mas}. See also \cite{ref_bib_li1} and \cite{ref_bib_li2} for $d=1$).
From Lemma \ref{ref_lemm_invert} we deduce 
\[
b_{\mu , \| \cdot \|_{\infty }}^{-1}(R) \sim (R/c)^{-\gamma } \text{ as } R \rightarrow \infty .
\]
Corollary \ref{ref_coro_mappsdfre} implies
\[
D_{\mu, r, \| \cdot \|_{\infty }}^{\infty }(R) \sim \left( \frac{c}{R}  \right)^{\gamma  r} e^{-R}
\text{ as } R \rightarrow \infty .
\]
From \cite[Corollary 1.3]{ref_bib_Graf_Luschgy_3} we deduce
\[
D_{\mu, r, \| \cdot \|_{\infty }}^{0 }(R) \approx \left( \frac{1}{R}  \right)^{\gamma  r} 
\text{ as } R \rightarrow \infty .
\]
In one dimension ($d=1$) we know (cf. relation (3.2) in \cite{ref_bib_Graf_Luschgy_3}) that  
\[
b_{\mu , \| \cdot \|_{p }}(s) \approx s^{-1/\gamma } \text{ as } s \rightarrow 0
\]
and (cf. \cite[Corollary 1.3]{ref_bib_Graf_Luschgy_3})
\[
D_{\mu, r, \| \cdot \|_{p }}^{0}(R) \approx \left( \frac{1}{R}  \right)^{\gamma  r} 
\text{ as } R \rightarrow \infty .
\]
Applying Remark \ref{ref_rem_approx} and Corollary \ref{ref_coro_imme}
we obtain
\[
D_{\mu, r, \| \cdot \|_{p }}^{\infty }(R) \lessapprox \left( \frac{1}{R}  \right)^{\gamma  r} e^{-R}
\text{ as } R \rightarrow \infty .
\]
Sharp asymptotics for $b_{\mu , \| \cdot \|_{p }}(\cdot )$ are known if $p=2$ 
(cf. \cite{ref_bib_bronski}). Thus a $c_{2}>0$ exists such that
\[
b_{\mu , \| \cdot \|_{2 }}(s) \sim c_{2} s^{-1/\gamma } \text{ as } s \rightarrow 0.
\]
As above, Lemma \ref{ref_lemm_invert} and Corollary \ref{ref_coro_mappsdfre} yields 
\[
D_{\mu, r, \| \cdot \|_{2 }}^{\infty }(R) \sim \left( \frac{c_{2}}{R}  \right)^{\gamma  r} e^{-R}
\text{ as } R \rightarrow \infty .
\]
For $\alpha \in \{ 0, 1 \}$ and $d=1$ we have (cf. \cite[Theorem 1.1.]{ref_bib_dereich}) 
\begin{equation}
\label{ref_err_opt_sim}
D_{\mu, r, \| \cdot \|_{\infty }}^{\alpha }(R) \sim \left( \frac{c_{0}(\gamma )}{R}  \right)^{\gamma  r}
\text{ as } R \rightarrow \infty 
\end{equation}
for some constant $c_{0}(\gamma ) \in {} ]0,\infty[$.
In view of Remark \ref{ref_rem_opt_quant} relation (\ref{ref_err_opt_sim}) is also true 
for all $\alpha \in {} ]0,1[$.
Moreover a $c_{p}(\gamma )$ exists such that
\[
D_{\mu, r, \| \cdot \|_{p}}^{\alpha }(R) \sim \left( \frac{c_{p}(\gamma )}{R}  \right)^{\gamma  r}
\text{ as } R \rightarrow \infty 
\]
for every $\alpha \in [0,1]$ (cf. \cite[Theorem 1.3.]{ref_bib_dereich}).
Independent of the norm of the Banach space $E$, the asymptotical order
of $D_{\mu, r}^{\alpha }( R )$ for large $R$ remains constant for $\alpha \in [0,1]$.
If $\alpha \in {} ]1, \infty [$, then the asymptotical order changes. 
The asymptotic can be determined by applying Corollary \ref{ref_coro_mappsdfrexx}.
We obtain
\[
D_{\mu, r, \| \cdot \|_{\infty }}^{\alpha }(R) \sim 
\left( \frac{\alpha }{ \alpha - 1} \cdot \frac{c }{R}  \right)^{\gamma  r} e^{-\frac{\alpha - 1}{\alpha }R}
\text{ as } R \rightarrow \infty 
\]
and
\[
D_{\mu, r, \| \cdot \|_{2 }}^{\alpha }(R) \sim 
\left( \frac{\alpha }{ \alpha - 1} \cdot \frac{c_{2}}{R}  \right)^{\gamma  r} e^{-\frac{\alpha - 1}{\alpha }R}
\text{ as } R \rightarrow \infty 
\]
for every $\alpha \in {} ]1, \infty [$.
\smallskip \\
Case 2: there is a non-unique minimum among $H_{1},..,H_{d}$. \\
Because the one-dimensional case has been already treated in case 1 we can assume w.l.o.g.
that $d \geq 2$. If $d \geq 3$, then the asymptotical order of $b_{\mu , \| \cdot \|_{\infty }}( \cdot )$
is not yet completely determined, even if all $H_{i}$ 
are equal. (cf. \cite{ref_bib_dunker} and the references therein). 
If $d=2$, then $H_{1}=H_{2}=H$ and we have
\[
b_{\mu , \| \cdot \|_{\infty }}( s ) \approx (1/s)^{1/H} \left( \log ( 1/s ) \right)^{1 + 1/H} 
\text{ as } s \rightarrow 0 
\]
(cf. \cite[Theorem 5.2.]{ref_bib_bel}, see also \cite{ref_bib_tala} for the case $H=1/2$). 
Remark \ref{ref_rem_approx}
implies
\[
b_{\mu , \| \cdot \|_{\infty }}^{-1}(R) \approx  
R^{-H} \left( \log ( R ) \right)^{H+1}
\text{ as } R \rightarrow \infty .
\]
Corollary \ref{ref_coro_imme} yields
\[
D_{\mu, r, \| \cdot \|_{\infty }}^{\infty }(R) \lessapprox ( R^{-H} \left( \log ( R ) \right)^{H+1} )^{r} e^{-R}
\text{ as } R \rightarrow \infty .
\]
From \cite[Corollary 1.3]{ref_bib_Graf_Luschgy_3} we know that
\[
D_{\mu, r, \| \cdot \|_{\infty }}^{0}(R) \approx ( R^{-H} \left( \log ( R ) \right)^{H+1} )^{r} 
\text{ as } R \rightarrow \infty .
\]
If $d \geq 2$ and $H_{i}=1/2$ for every $i=1,..,d$ we know (cf. \cite{ref_bib_csaki}) that
\[
b_{\mu , \| \cdot \|_{2 }}( s ) \sim c_{d} (1/s)^{2} \left( \log ( 1/s ) \right)^{2d-2} \text{ as } 
s \rightarrow 0
\]
with $c_{d}=2^{d-2}/\left( \sqrt{2} \pi^{d-1} (d-1)! \right)$.
Lemma \ref{ref_lemm_invert} implies
\[
b_{\mu , \| \cdot \|_{2 }}^{-1}( R ) \sim c_{d}^{1/2} 
2^{-(d-1)} R^{-1/2} \left( \log (R) \right)^{d-1}
\text{ as } R \rightarrow \infty .
\]
Corollary \ref{ref_coro_mappsdfre} yields
\[
D_{\mu, r, \| \cdot \|_{2 }}^{\infty }(R) \sim 
( c_{d}^{1/2} 2^{-(d-1)} R^{-1/2} \left( \log (R) \right)^{d-1}  )^{r} e^{-R}
\text{ as } R \rightarrow \infty .
\]
From \cite[Corollary 1.3]{ref_bib_Graf_Luschgy_3} we obtain 
\[
D_{\mu, r, \| \cdot \|_{2 }}^{0 }(R) \approx ( R^{-1/2} \left( \log ( R ) \right)^{d-1} )^{r} 
\text{ as } R \rightarrow \infty .
\]
Moreover we know (cf. \cite{ref_bib_luschg}, relation (3.13)) that
\begin{equation}
\label{ref_equ_dmue2}
D_{\mu, 2, \| \cdot \|_{2 }}^{\alpha }(R) \sim  (b_{d}  R^{-1/2} \left( \log ( R ) \right)^{d-1} )^{2}
\text{ as } R \rightarrow \infty
\end{equation}
with $b_{d} = \sqrt{2} / ( \pi^{d} (d-1)! )$ and $\alpha = 0$. 
From \cite[Theorem 2.2.]{ref_bib_luschg} and \cite[Theorem 1.1.]{ref_bib_Graf_Luschgy_4} 
we deduce that (\ref{ref_equ_dmue2}) also holds for $\alpha = 1$. 
In view of Remark \ref{ref_rem_opt_quant} relation (\ref{ref_equ_dmue2}) is also true 
for all $\alpha \in {} ]0,1[$.
As in Case 1 
the asymptotical order of $D_{\mu, r}^{\alpha }( R )$ for large $R$ and $\alpha > 1$ is 
different from the one for $\alpha \in [0,1]$. 
It can be determined by applying Corollary \ref{ref_coro_mappsdfrexx}.
We get
\[
D_{\mu, r, \| \cdot \|_{2 }}^{\alpha }(R) \sim 
\left( c_{d}^{1/2} 2^{-(d-1)} \left(\frac{\alpha - 1}{\alpha } R\right)^{-1/2} 
\left( \log \left(\frac{\alpha - 1}{\alpha } R\right) \right)^{d-1}  \right)^{r} 
e^{-\frac{\alpha - 1}{\alpha } R} 
\]
as $R \rightarrow \infty$ for every $\alpha \in {} ]1, \infty [$.

\subsection{L\'{e}vy fractional Brownian motion}

The L\'{e}vy fractional Brownian motion of order $H \in {} ]0,1[$ is 
a centered Gaussian process defined by
\[
X_{0}=0, \quad \mathbb{E} ( ( X_{t} - X_{s} )^{2} ) = \|t-s\|^{2H } \text{ for } s,t \in [0,1]^{d},
\]
if $\| \cdot \|$ denotes the Euclidean norm on $\mathbb{R}^{d}$. 
For this stochastic process we have
\[
b_{\mu , \| \cdot \|_{\infty }}( s ) \approx (1/s)^{d/H} \text{ as } s \rightarrow 0
\]
(cf. \cite{ref_bib_shao}). 
Remark \ref{ref_rem_approx}
yields 
\[
b_{\mu , \| \cdot \|_{\infty }}^{-1}( R ) \approx (1/R)^{H/d} \text{ as } R \rightarrow \infty .
\]
For $\alpha > 1$ Corollary \ref{ref_coro_imme} and Lemma \ref{ref_lemm_conj22} implies
\[
D_{\mu, r, \| \cdot \|_{\infty }}^{\alpha }(\frac{\alpha }{\alpha - 1}R) \leq
D_{\mu, r, \| \cdot \|_{\infty }}^{\infty }(R) \lessapprox ( R^{-H/d}  )^{r} e^{-R}
\text{ as } R \rightarrow \infty .
\]
Applying \cite[Corollary 1.3]{ref_bib_Graf_Luschgy_3} we obtain 
\[
D_{\mu, r, \| \cdot \|_{\infty }}^{0 }(R) \approx ( R^{-H/d}  )^{r} 
\text{ as } R \rightarrow \infty .
\]

\subsection{$m-$times integrated Brownian Motion, Fractional Integrated Brownian Motions, 
$m$-integrated Brownian sheet}

For $\beta > 0$ we define the centered Gaussian probability vector 
\[
X^{\beta}=(X_{t}^{\beta})_{t \in [0,1]}
\]
by
\[
X_{t}^{\beta } = \frac{1}{\Gamma (\beta ) } \int_{0}^{t} (t-s)^{\beta - 1} B_{s} ds, \qquad t \in [0,1], 
\]
where $B_{s}$ denotes Brownian motion. Since 
a $c(\beta ) \in {} ]0, \infty[$ exists, such that 
\[
b_{\mu , \| \cdot \|_{\infty }}(s) \sim c(\beta ) s^{-2/(2 \beta + 1)} \text{ as } s \rightarrow 0
\]
(cf. \cite{ref_bib_li1} and \cite{ref_bib_li2}) we deduce from Lemma \ref{ref_lemm_invert} that 
\[
b_{\mu , \| \cdot \|_{\infty }}^{-1}(R) \sim (R/c(\beta ))^{-(\beta + 1/2)} \text{ as } R \rightarrow \infty .
\]
Together with Corollary \ref{ref_coro_mappsdfre} we obtain
\[
D_{\mu, r, \| \cdot \|_{\infty }}^{\infty }(R) \sim \left( \frac{c(\beta )}{R}  \right)^{(\beta + 1/2) r} e^{-R}
\text{ as } R \rightarrow \infty .
\]
Corollary \ref{ref_coro_mappsdfrexx} yields
\[
D_{\mu, r, \| \cdot \|_{\infty }}^{\alpha }(R) \sim 
\left( \frac{\alpha}{\alpha - 1} \cdot \frac{c(\beta )}{R}  \right)^{(\beta + 1/2) r} 
e^{-\frac{\alpha - 1}{\alpha }R}
\text{ as } R \rightarrow \infty 
\] 
for every $\alpha \in {} ]1, \infty [$.
Moreover we have
\[
D_{\mu, r, \| \cdot \|_{\infty }}^{0 }(R) \approx \left( \frac{1}{R}  \right)^{(\beta + 1/2) r} 
\text{ as } R \rightarrow \infty .
\]
(cf. \cite[p. 1059]{ref_bib_Graf_Luschgy_3}).
If $\beta = m \in \mathbb{N}$, then a $c(m)>0$ exists, such that 
\[
b_{\mu , \| \cdot \|_{2}}(s) \sim c( m ) s^{-2/(2 m + 1)} \text{ as } s \rightarrow 0 
\]
(cf. \cite[Theorem 1.1]{ref_bib_chen}).
Again, Lemma \ref{ref_lemm_invert} and Corollary \ref{ref_coro_mappsdfre} are implying that
\[
D_{\mu, r, \| \cdot \|_{2}}^{\infty }(R) \sim \left( \frac{c(m )}{R}  \right)^{(m + 1/2) r} e^{-R} 
\text{ as } R \rightarrow \infty .
\]
Applying \cite[Corollary 1.3]{ref_bib_Graf_Luschgy_3} we deduce
\[
D_{\mu, r, \| \cdot \|_{2}}^{0 }(R) \approx \left( \frac{1}{R}  \right)^{(m + 1/2) r} \text{ as } R \rightarrow \infty .
\]
Moreover we know (cf. \cite{ref_bib_luschg}, relation (3.7)) that
\begin{equation}
\label{ref_equ_dmue3}
D_{\mu, 2, \| \cdot \|_{2 }}^{\alpha }(R) \sim  \left( \frac{c_{0}(m)}{R}  \right)^{(m + 1/2) r}
\text{ as } R \rightarrow \infty
\end{equation}
with a $c_{0}(m) \in {} ]0, \infty[$ and $\alpha = 0$. 
From \cite[Theorem 2.2.]{ref_bib_luschg} and \cite[Theorem 1.1.]{ref_bib_Graf_Luschgy_4} 
we deduce that (\ref{ref_equ_dmue3}) also holds for $\alpha = 1$. 
In view of Remark \ref{ref_rem_opt_quant} relation (\ref{ref_equ_dmue3}) is also true 
for all $\alpha \in {} ]0,1[$.
If $\alpha \in {} ]1,\infty [$, then Corollary \ref{ref_coro_mappsdfrexx} yields the error asymptotics.
We deduce
\[
D_{\mu, r, \| \cdot \|_{2}}^{\alpha }(R) \sim 
\left( \frac{\alpha}{ \alpha - 1} \frac{c(m )}{ R}  \right)^{(m + 1/2) r} 
e^{-\frac{\alpha - 1}{\alpha } R} 
\text{ as } R \rightarrow \infty 
\]
for every $\alpha \in {} ]1,\infty [$.
Results for small ball asymptotics of more general $m-$times 
integrated Brownian motions can be found in \cite{ref_bib_gao}
and \cite{ref_bib_naza}.

Now let $m \in \mathbb{N}$ and $(B_{t})_{t \in [0,1]^{d}}$ be a $d-$dimensional Brownian sheet, i.e.
$(B_{t})$ is a centered Gaussian measure characterized through the covariance function
\[
\mathbb{E} ( B_{s} B_{t} ) = \prod_{j=1}^{d} \min ( s_{j}, t_{j} ) 
\]
for $s=(s_{1},..,s_{d}) \in [0,1]^{d}$ and $t=(t_{1},..,t_{d}) \in [0,1]^{d}$. The $m-$integrated Brownian
sheet $(X_{t})_{t \in \mathbb{R}^{d}}$ is now defined by
\[
X_{m}(t) = \int_{0}^{t_{1}} \cdots \int_{0}^{t_{d}} \prod_{j=1}^{d} \frac{(t_{j}-u_{j})^{m}}{m !} B(du_{1},..,du_{d}).
\]
For this process a $c=c(m,d)>0$ exists, such that
\[
b_{\mu , \| \cdot \|_{2}}(s) \sim c s^{-2/(2 m + 1)} \left( \log (1/s) \right)^{\frac{(d-1)(2m+2)}{2m+1}} 
\text{ as } s \rightarrow 0 
\]
(cf. \cite[Corollary 5.2]{ref_bib_fill}).
Lemma \ref{ref_lemm_invert} yields
\[
b^{-1}_{\mu , \| \cdot \|_{2}}(R) \sim  
\left( \frac{c}{R} \right)^{m+1/2}  \left( \frac{\log (R)}{m+1/2} \right)^{(d-1)(m+1)}
\text{ as } R \rightarrow \infty .
\]
Corollary \ref{ref_coro_mappsdfre} implies 
\[
D_{\mu, r, \| \cdot \|_{2 }}^{\infty }(R) \sim 
\left(  \left( \frac{c}{R} \right)^{m+1/2}  \left( \frac{\log (R)}{m+1/2} \right)^{(d-1)(m+1)}  
\right)^{r} e^{-R} \text{ as } R \rightarrow \infty .
\]
On the other hand we deduce from \cite[Corollary 1.3]{ref_bib_Graf_Luschgy_3} that
\[
D_{\mu, r, \| \cdot \|_{2 }}^{0 }(R) \approx \left(  R^{-(m + 1/2)} \left( \log (R) \right)^{(d-1)(m+1)}  \right)^{r}
\text{ as } R \rightarrow \infty .
\] 
In case of $\alpha > 1$ we obtain sharp asymptotics for the optimal quantization error by
Corollary \ref{ref_coro_mappsdfrexx}.
We have
\[
D_{\mu, r, \| \cdot \|_{2 }}^{\alpha }(R) \sim 
\left(  \left( \frac{\alpha}{\alpha  - 1} \frac{c}{R} \right)^{m+1/2}  
\left( \frac{\log (\frac{\alpha - 1}{\alpha } R)}{m+1/2} \right)^{(d-1)(m+1)}  \right)^{r} 
e^{-\frac{\alpha - 1}{\alpha } R}  
\]
as $R \to \infty$ for every $\alpha \in {} ]1, \infty[$.

\subsection{Fractional Ornstein-Uhlenbeck Processes}

Let $\gamma > 0$ and $H \in {} ]0,2[$.
Let us consider the centered stationary fractional 
Ornstein-Uhlenbeck process, which is a Gaussian process defined by the covariance function
\[
\mathbb{E} X^{H}_{t} X^{H}_{s} = e^{- \gamma | t-s |^{H}}, \qquad t,s \in [0,1].
\]
Since 
\[
b_{\mu , \| \cdot \|_{\infty}}(s) \approx s^{-2/H} \text{ as } s \rightarrow 0
\]
(cf. \cite[Theorem 2.1]{ref_bib_monrad})
and
\[
b_{\mu , \| \cdot \|_{p}}(s) \approx s^{-2/H} \text{ as } s \rightarrow 0
\]
(cf. \cite[p. 1061]{ref_bib_Graf_Luschgy_3})
we deduce from Remark \ref{ref_rem_approx} and Corollary \ref{ref_coro_imme}
together with Lemma \ref{ref_lemm_conj22} for $\alpha > 1$ that
\[
D_{\mu, r, \| \cdot \|_{\infty }}^{\alpha }(\frac{\alpha }{\alpha - 1}R) \leq
D_{\mu, r, \| \cdot \|_{\infty }}^{\infty }(R) \lessapprox \left(   R^{- H / 2}  \right)^{r} e^{-R}
\text{ as } R \rightarrow \infty 
\]
and
\[
D_{\mu, r, \| \cdot \|_{p }}^{\alpha }(\frac{\alpha }{\alpha - 1}R) \leq
D_{\mu, r, \| \cdot \|_{p }}^{\infty }(R) \lessapprox \left(   R^{- H / 2}  \right)^{r} e^{-R} 
\text{ as } R \rightarrow \infty .
\]
Moreover (cf. \cite{ref_bib_Graf_Luschgy_3})
\[
D_{\mu, r, \| \cdot \|_{\infty }}^{0 }(R) \approx \left(   R^{- H / 2}  \right)^{r} 
\text{ as } R \rightarrow \infty .
\]
and
\[
D_{\mu, r, \| \cdot \|_{p }}^{0 }(R) \approx \left(   R^{- H / 2}  \right)^{r} 
\text{ as } R \rightarrow \infty .
\]
If $H = 1$, then we have the standard Ornstein-Uhlenbeck process which can also be defined 
as the solution of a stochastic differential equation. From this special case we can also
generalize the standard Ornstein-Uhlenbeck process to Gaussian diffusions, defined as a solution
of a certain stochastic differential equation. Asymptotic small ball probabilities for such processes were
derived by Fatalov \cite{ref_bib_fatalov}.
For results about the asymptotics of the optimal quantization error for 
such diffusions and $\alpha \in \{ 0, 1\}$
the reader is referred to Dereich \cite{ref_bib_dereich2, ref_bib_dereich3} 
resp. Luschgy and Pag\`es \cite{ref_bib_luschgy, ref_bib_luschgy2}.
The optimal quantization of Fractional Ornstein-Uhlenbeck Processes with higher 
dimensional index space has been discussed by Luschgy and Pag\`es \cite{ref_bib_luschg}.
Once again we observe the change in the asymptotical order of the high rate asymptotics
of the optimal quantization error, if the entropy index $\alpha$ becomes larger than $1$.

\begin{remark}
Taking the sum $Z=X+Y$ of two not necessarily independent joint Gaussian random vectors $X,Y$
it is possible to determine the asymptotical order of the small ball probability of $Z$, if
this order is known for $X$ and $Y$ (cf. \cite[Theorem 2.1]{ref_bib_el_nouty2}). 
Moreover, small ball probabilities of fractional mixtures of fractional Gaussian measures are
investigated by El-Nouty \cite{ref_bib_el_nouty, ref_bib_el_nouty2} .
\end{remark}

\subsection{Slepian Gaussian fields}

Let $a=(a_{1},..,a_{d}) \in {} ]0,\infty [^{d}$.
We consider the centered Gaussian process $(X_{t})_{t \in [0,1]^{d}}$ characterized by the 
covariance function
\[
\mathbb{E}(X_{t}X_{s}) = \prod_{i=1}^{d} \max ( 0, a_{i} - | s_{i} - t_{i} | ) .
\]
For this process we have
\[
b_{\mu , \| \cdot \|_{\infty }}(s) \approx (1/s)^{2} ( \log ( 1/s ) )^{3} \text{ as } s \rightarrow 0
\]
(cf. \cite[Theorem 1.1]{ref_bib_gao_li}) and
\[
b_{\mu , \| \cdot \|_{2 }}(s) \approx (1/s)^{2} ( \log ( 1/s ) )^{2d-2} \text{ as } s \rightarrow 0
\]
(cf. \cite[Theorem 1.1]{ref_bib_gao_li}).
Thus we obtain from Remark \ref{ref_rem_approx}, Corollary \ref{ref_coro_imme} 
together with Lemma \ref{ref_lemm_conj22} for $\alpha > 1$ that
\[
D_{\mu, r, \| \cdot \|_{\infty }}^{\alpha }(\frac{\alpha }{\alpha - 1}R) \leq
D_{\mu, r, \| \cdot \|_{\infty }}^{\infty }(R) \lessapprox \left(   R^{-1/2} (\log (R) )^{3/2}  \right)^{r} e^{-R}
\text{ as } R \rightarrow \infty 
\]
and
\[
D_{\mu, r, \| \cdot \|_{2 }}^{\alpha }(\frac{\alpha }{\alpha - 1}R) \leq
D_{\mu, r, \| \cdot \|_{2 }}^{\infty }(R) \lessapprox \left(   R^{-1/2} (\log (R) )^{d-1}  \right)^{r} e^{-R} 
\text{ as } R \rightarrow \infty .
\]
Moreover (cf. \cite{ref_bib_Graf_Luschgy_3})
\[
D_{\mu, r, \| \cdot \|_{\infty }}^{0 }(R) \approx \left(   R^{-1/2} (\log (R) )^{3/2}  \right)^{r} 
\text{ as } R \rightarrow \infty 
\]
and
\[
D_{\mu, r, \| \cdot \|_{2 }}^{0 }(R) \approx \left(   R^{-1/2} (\log (R) )^{d-1}  \right)^{r} 
\text{ as } R \rightarrow \infty .
\]
Finally, also this class of Gaussian processes shows the change in the optimal quantization error asymptotics as 
$\alpha$ increases.



\begin{thebibliography}{9}


\bibitem{ref_bib_Aczel} Acz\'el, J., Dar\'oczy, Z.:
On Measures of Information and Their Characterizations,
Mathematics in Science and Engineering, vol. 115.
Academic Press, London (1975) 

\bibitem{ref_bib_and}
Anderson, T.W.: 
The integral of a symmetric convex set and some probability inequalities.
Proc. Amer. Math. Soc. \textbf{6}, 170–-176 (1955) 

\bibitem{ref_bib_beck} Beck, C., Schl\"{o}gl, F.: 
Thermodynamics of chaotic systems.
Cambridge University Press, Cambridge (1993) 

\bibitem{ref_bib_Behara} Behara, M.:
Additive and nonadditive measures of entropy.
John Wiley, New Delhi (1990)

\bibitem{ref_bib_bel} 
Belinsky, E., Linde, W.: 
Small Ball Probabilities of Fractional Brownian Sheets via
Fractional Integration Operators.
J. Theoret. Probab. \textbf{15}, 589-612 (2002)

\bibitem{ref_bib_Ber72}
Berger, T.:
Optimum quantizers and permutation codes.
IEEE Trans. Inform. Theory \textbf{18}, 759--765 (1972)

\bibitem{ref_bib_bingham} Bingham, N.H., Goldie, C.M., Teugels, J.L.:
Regular variation.
Encyclopedia of Mathematics and its applications,
Cambridge University Press, Cambridge (1987)

\bibitem{ref_bib_bogachev} Bogachev, I.V.: 
Gaussian Measures. AMS (1998) 

\bibitem{ref_bib_bronski} 
Bronski, J.C.:
Small Ball Constants and Tight Eigenvalue Asymptotics for
Fractional Brownian Motions.
J. Theoret. Probab. \textbf{16}, 87-100 (2003)

\bibitem{ref_bib_chen}
Chen, X., Li, W.V.: 
Quadratic functionals and small ball probabilities for the 
$m-$fold integrated Brownian Motion.
Ann. Probab. \textbf{31}, 1052-1077 (2003)

\bibitem{ref_bib_csaki}
Cs\'aki, E.: 
On small values of the square integral of a multiparameter Wiener process.
Statistics and Probability. 
Proc. 3rd Pannonian Symp., Visegr\'ad/Hung. 1982, 19-26 (1984) 

\bibitem{ref_bib_de}
Dereich, S.: High resolution coding of stochastic processes and small ball probabilities.
Ph.D. Dissertation (2003) 

\bibitem{ref_bib_der}
Dereich, S., Fehringer, F., Matoussi, A., Scheutzow, M.: On the link between small
ball probabilities and the quantization problem for Gaussian measures on Banach spaces. 
J. Theor. Probab. \textbf{16}, 249-265 (2003)

\bibitem{ref_bib_dereich} Dereich, S., Scheutzow, M.: 
High-resolution quantization and entropy coding for fractional
Brownian motion. Electron. J. Probab. \textbf{11}, 700-722 (2006)

\bibitem{ref_bib_dereich2} Dereich, S.:
The coding complexity of diffusion processes under supremum norm distortion.
Stochastic Process. Appl. \textbf{118}, 917-937 (2008)

\bibitem{ref_bib_dereich3} Dereich, S.: 
The coding complexity of diffusion processes under $L^p[0,1]$-norm distortion.
Stochastic Process. Appl. \textbf{118}, 938-951 (2008) 

\bibitem{ref_bib_dunker}
Dunker, T.: 
Estimates for the Small Ball Probabilities of the Fractional Brownian Sheet.
J. Theor. Probab. \textbf{13}, 357-382 (2000)

\bibitem{ref_bib_el_nouty} El-Nouty, C.: 
The fractional mixed fractional Brownian motion and
fractional Brownian sheet.
ESAIM, Probab. Stat. \textbf{11}, 448-465 (2007)

\bibitem{ref_bib_el_nouty2} El-Nouty, C.:
On the lower classes of some mixed fractional Gaussian processes with two logarithmic factors.
J. Appl. Math. Stochastic Anal. (2008)
doi:10.1155/2008/160303

\bibitem{ref_bib_fatalov} Fatalov, V.R.: 
Exact asymptotics of small deviation for a stationary Ornstein-Uhlenbeck process
and some Gaussian diffusion processes in the $L^{p}$-norm, $2 \leq p \leq \infty$.
Probl. Inf. Transm. \textbf{44}, 138-155 (2008)

\bibitem{ref_bib_fill}
Fill, J.A., Torcaso, F.: 
Asymptotic analysis via Mellin transforms for small deviations in 
$L^{2}-$norm of integrated Brownian sheets.
Probab. Theory Relat. Fields. \textbf{130}, 259-288 (2004)

\bibitem{ref_bib_gao}
Gao, F., Hannig, J., Torcaso, T.: 
Integrated Brownian Motions and Exact $L_{2}-$Small Balls.
Ann. Probab. \textbf{31}, 1320-1337 (2003)

\bibitem{ref_bib_gao_li}
Gao, F., Li, W.V.: 
Small ball probabilities for the Slepian Gaussian fields.
Trans. Am. Math. Soc. 
\textbf{359}, 1339-1350 (2007)

\bibitem{ref_bib_Graf_Luschgy_1} Graf, S., Luschgy, H.:
Foundations of Quantization for Probability Distributions.
Lecture Notes 1730 Springer (2000) 

\bibitem{ref_bib_Graf_Luschgy_3} Graf, S., Luschgy, H., Pag\`es, G.:  
Functional quantization and small ball probabilities for Gaussian processes. 
J. Theor. Probab. \textbf{16}, 1047-1062 (2003)

\bibitem{ref_bib_Graf_Luschgy_4} Graf, S., Luschgy, H.: 
Entropy-constrained functional quantization of Gaussian measures. 
Proc. Am. Math. Soc. \textbf{133}, 3403–-3409 (2005) 

\bibitem{ref_bib_Gray}
Gray, R.M.,  Linder, T., Li, J.:
A Lagrangian formulation of Zador's entropy-constrained quantization theorem. 
IEEE Trans. Inform. Theory \textbf{48}, 695--707 (2002) 

\bibitem{ref_bib_Gyo00}
Gy\"{o}rgy, A.,  Linder, T.:
Optimal entropy-constrained scalar quantization of a
uniform source. 
IEEE Trans. Inform. Theory \textbf{46}, 2704--2711 (2000) 

\bibitem{ref_bib_kreit2} 
Kreitmeier, W.:
Optimal Quantization for the one-dimensional uniform distribution 
with R\'enyi-$\alpha$-entropy constraints.
Kybernetika \textbf{46}, 96--113 (2010) 

\bibitem{ref_bib_kreit3} 
Kreitmeier, W.:
Error bounds for high-resolution quantization with
R\'enyi-$\alpha$-entropy constraints.
Acta Math. Hungar. \textbf{127}, 34-51 (2010)

\bibitem{ref_bib_kreit4}
Kreitmeier, W.:
Optimal vector quantization in terms of Wasserstein distance.
J. Multivariate Anal. \textbf{102}, 1225-1239 (2011) 

\bibitem{ref_bib_kreit5} 
Kreitmeier, W., Linder, T.:
High-resolution scalar quantization with R\'enyi entropy constraint.
IEEE Trans. Inform. Theory \textbf{57}, 6837--6859 (2011) 

\bibitem{ref_bib_li1} 
Li, W.V., Linde, W.: 
Existence of small ball constants for fractional Brownian motions.
C. R. Acad. Sci. Paris S\'er. I Math. \textbf{326}, 1329-1334 (1998) 

\bibitem{ref_bib_li2} 
Li, W.V., Linde, W.:
Approximation, metric entropy and small ball estimates for Gaussian measures.
Ann. Probab. \textbf{27}, 1556-1578 (1999)

\bibitem{ref_bib_lifs} 
Lifshits, M., Simon, T.: 
Small deviations for fractional stable processes.
Ann. Inst. Henri Poincar\'e, Probab. Stat. \textbf{41}, 725-752 (2005)

\bibitem{ref_bib_luschg}
Luschgy, H., Pag\`es, G.: 
Sharp asymptotics of the functional quantization problem for
Gaussian processes. 
Ann. Probab. \textbf{32}, 1574-1599 (2004)

\bibitem{ref_bib_luschgy}
Luschgy, H., Pag\`es, G.:
Functional quantization of a class of Brownian diffusions: a constructive approach.
Stochastic Processes Appl. \textbf{116}, 310-336 (2006)

\bibitem{ref_bib_luschgy2}
Luschgy, H., Pag\`es, G.:
Functional quantization rate and mean pathwise regularity of processes with an application to L\'evy processes.
Ann. Appl. Probab. \textbf{18}, 427-469 (2008)

\bibitem{ref_bib_mas} 
Mason, D.M., Shi, Z.:
Small Deviations for Some Multi-Parameter Gaussian Processes.
J. Theor. Probab. \textbf{14}, 213-239 (2001)

\bibitem{ref_bib_monrad}
Monrad, D., Rootz\'en, H.: 
Small values of Gaussian processes and functional laws
of the iterated logarithm.
Probab. Theory Relat. Fields \textbf{101}, 173-192 (1995)

\bibitem{ref_bib_naza}
Nazarov, A.I.: 
On the Sharp Constant in the Small Ball Asymptotics of Some Gaussian Processes
under $L_{2}-$Norm.
J. Math. Sci., New York \textbf{117}, 4185-4210 (2005)

\bibitem{ref_bib_nicu}
Niculescu, C., Persson, L.: 
Convex functions and their applications. A contemporary approach,
CMS Books in Mathematics, Springer, New York (2005) 

\bibitem{ref_bib_shao}
Shao, Q.M., Wang, D.:  
Small ball probabilities of Gaussian fields.
Probab. Theory Relat. Fields \textbf{102}, 511-517 (1995)

\bibitem{ref_bib_tala} 
Talagrand, M.: 
The small ball problem for the Brownian sheet.
Ann. Probab. \textbf{22}, 1331-1354 (1994)

\end{thebibliography}
\end{document}